\documentclass{amsart}

\newtheorem{theorem}{Theorem}[section]
\newtheorem{lemma}[theorem]{Lemma}
\newtheorem{corollary}[theorem]{Corollary}

\theoremstyle{definition}
\newtheorem{definition}[theorem]{Definition}
\newtheorem{example}[theorem]{Example}

\theoremstyle{remark}
\newtheorem{remark}[theorem]{Remark}

\usepackage{amscd}
\usepackage{amsmath}
\usepackage{amsbsy}
\usepackage{amsfonts}
\usepackage{amssymb}
\usepackage{amsthm}
\usepackage[mathscr]{eucal}
\usepackage[OT2,T1]{fontenc}
\usepackage[T3,T4]{fontenc}
\usepackage{latexsym}
\usepackage{enumerate}
\usepackage{graphicx}
\usepackage{color}	
\usepackage{epsfig}
\usepackage[numbers]{natbib}
\usepackage{listings}
\usepackage{manfnt}
\usepackage{psfrag}
\usepackage{verbatim}
\usepackage{url}
\usepackage{undertilde}
\usepackage{fancyhdr}

\newcommand{\Z}{\mathbb{Z}}
\newcommand{\Q}{\mathbb{Q}}

\newcommand{\smatr}[4]{\Bigl( \begin{smallmatrix} 
             #1 & #2 \\ #3 & #4 \end{smallmatrix} \Bigr) }
\newcommand{\Red}[1]{\widetilde{#1}}

\newcommand{\svect}[2]{\Bigl( \begin{smallmatrix}{#1}\\{#2}\end{smallmatrix}\Bigr) }

\newcommand{\Pt}[1]{\mathscr{#1}}
\newcommand{\Padd}{\boxplus}
\newcommand{\Psub}{\boxminus}
\DeclareSymbolFont{cyrls}{OT2}{wncyr}{m}{n}
\DeclareMathSymbol{\X}{\mathalpha}{cyrls}{"48}
\DeclareMathSymbol{\Y}{\mathalpha}{cyrls}{"55}
\DeclareMathSymbol{\Si}{\mathalpha}{cyrls}{"62}
\DeclareMathSymbol{\Sha}{\mathalpha}{cyrls}{"58}

\DeclareSymbolFont{cows}{T4}{fcr10}{m}{n}
\DeclareMathSymbol{\EB}{\mathalpha}{cows}{"80}
\DeclareMathSymbol{\ED}{\mathalpha}{cows}{"81}
\DeclareMathSymbol{\EF}{\mathalpha}{cows}{"84}
\DeclareMathSymbol{\EK}{\mathalpha}{cows}{"88}
\DeclareMathSymbol{\EN}{\mathalpha}{cows}{"89}
\DeclareMathSymbol{\EP}{\mathalpha}{cows}{"91}
\DeclareMathSymbol{\ET}{\mathalpha}{cows}{"9B}
\DeclareMathSymbol{\EU}{\mathalpha}{cows}{"8E}
\DeclareMathSymbol{\ML}{\mathalpha}{cows}{"4C}
\DeclareMathSymbol{\MM}{\mathalpha}{cows}{"4D}
\DeclareMathSymbol{\MN}{\mathalpha}{cows}{"4E}
\DeclareMathSymbol{\MS}{\mathalpha}{cows}{"53}
\DeclareMathSymbol{\MT}{\mathalpha}{cows}{"54}

\DeclareSymbolFont{cheese}{T3}{tipa10}{m}{n}
\DeclareMathSymbol{\Em}{\mathalpha}{cheese}{"4D}
\DeclareMathSymbol{\Ef}{\mathalpha}{cheese}{"EA}
\DeclareMathSymbol{\Eg}{\mathalpha}{cheese}{"A5}
\DeclareMathSymbol{\EG}{\mathalpha}{cheese}{"E5}

\begin{document}

\title[Proof of $\text{Cl}^+(\Q(\sqrt{\Delta }))^2 \simeq \text{ker}\bigl( H^1(G,\mathcal{P}(\overline{\Z})) \rightarrow H^1(G,\mathcal{P}(\overline{\Q})) \bigr)$]{An isomorphism between the narrow ideal class group of squared ideals of a quadratic number field and the kernel of a homomorphism between cohomology groups for Pell conics}

\author[S. Hambleton]{Samuel A. Hambleton}

\address{School of Mathematics and Physics, University of Queensland, St. Lucia, Queensland, Australia 4072}

\email{sah@maths.uq.edu.au}

%    General info

\date{August 08, 2011.}


\begin{abstract}
Two proofs are provided that the narrow ideal class group of squared ideals of a quadratic number field is isomorphic to the kernel of a homomorphism between cohomology groups for Pell conics, Lemmermeyer's obstruction to descent for Pell conics. These proofs make use of a particular subgroup of a Pell conic over algebraic numbers.
\end{abstract}

\maketitle

\section{Introduction}

Let $E$ be an elliptic curve. The Tate-Shafarevich group for $E(\Q)$ measures how difficult it may be to find a rational point of $E$ using the process of descent, see \cite{Milne,Silverman}. This group is defined in terms of the kernel of a homomorphism between cohomology groups
\begin{equation}\label{ShaEll}
\Sha (E/ \Q ) = \text{ker}\Bigl( H^1(G, E(\overline{\Q})) \longrightarrow \prod_{p \in {\text{ primes}}} H^1(G, E(\overline{\Q}_p)) \Bigr) .
\end{equation}
Non-trivial elements of $\Sha (E/ \Q )$ correspond to principal homogeneous spaces for $E$ over $\Q$ which have a $p$-adic point for every prime $p$, but no rational point. We will define the Tate-Shafarevich group for Pell conics, $\Sha (\mathcal{P}/ \Z )$, in a similar way to Eqn. \eqref{ShaEll} and show that our definition is consistent with Lemmermeyer's result \cite{LemNewBook} that $\Sha (\mathcal{P}/ \Z ) \simeq \text{Cl}^+(\Q(\sqrt{\Delta }))^2$, the narrow class group of binary quadratic forms $Q$ of discriminant $\Delta$, for which $Q = Q' \cdot Q'$ where $\cdot$ is composition of forms.

By the Hasse-Minkowski theorem, see \cite{Serre}, for a quadratic polynomial $$f(x, y) = a_{20} x^2 +a_{11} x y + a_{02} y^2 + a_{10} x + a_{01} y +a_{00}$$ with rational coefficients, there exists a rational solution to 
\begin{equation}\label{HasMin}
f(x, y) = 0 ,
\end{equation}
if and only if there exists a real solution to Eqn. \eqref{HasMin} and a solution in the field of $p$-adic numbers $\Q_p$ for every prime $p$. This provides a means of deciding whether the binary quadratic form 
\begin{equation}\label{intrbinform}
Q'(t,u) = A' t^2 + B t u + C u^2 
\end{equation}
represents $1$ in rational numbers, and finding such a point is generally easy. After defining abelian groups on the rational, algebraic and $p$-adic points of the conic $C: Q'(t, u) = 1$, with $Q'$ as in Eqn. \eqref{intrbinform}, it can be shown that the Hasse-Minkowski theorem implies that $$\text{ker}\Bigl( H^1(G, E(\overline{\Q})) \longrightarrow \prod_{p \in {\text{ primes}}} H^1(G, E(\overline{\Q}_p)) \Bigr)$$ is always trivial, however it says nothing about whether, given there exists a rational point $q = (t, u)$ satisfying 
\begin{equation}\label{garrrv}
Q'(t,u) = 1,
\end{equation}
there exists an integer representation of $1$ by $Q'$, Eqn. \eqref{intrbinform}. 

Let $\Delta $ be the fundamental discriminant of a quadratic number field $K = \Q(\sqrt{\Delta })$. This means that 
\begin{equation} 
\Delta =  
\left\{\begin{array}{ll} 
d = 1 + 4 \Em  &\mbox{if $d \equiv 1 \pmod{4}$} \\
4 d = 4 \Em &\mbox{if $d \equiv 2,3 \pmod{4}$} 
\end{array}\right. 
\end{equation} 
where $d$ is a square-free integer. We let $\Si \in \{ 0, 1\}$ be defined by
\begin{equation}
\Delta = \Si + 4 \Em .
\end{equation}
Pell conics \cite{Lem03,LemDes3,MCP} are affine curves of genus $0$ given by either of the equations\footnote{Completing the square transforms Eqn. \eqref{nrmform} into Eqn. \eqref{tracform}.} 
\begin{eqnarray}
\label{nrmform} \mathcal{P} : x^2 +\Si x y - \Em y^2 = 1 , \\
\label{tracform} \mathcal{C} : \X^2 -\Delta \Y^2 = 4 .
\end{eqnarray}
For the purposes of this article it is more appropriate to consider Pell conics in the form $\mathcal{P}$ due to the necessity to use binary quadratic forms. It is convenient to introduce the matrix
\begin{equation}\label{addmatrix} 
\MN_{\Pt{P}} = \smatr{x}{\Em y}{y}{x + \Si y} .
\end{equation}
Incidentally, Pell conics may be expressed as $\mathcal{P} : \text{det}(\MN_{\Pt{P}}) = 1$. The points of a Pell conic, defined over various rings and fields, form an abelian group with the binary operation $\Padd$ for $\mathcal{P}$ given by 
\begin{align}
\label{grplawnm}  \Pt{P}_1 \Padd \Pt{P}_2 & = \MN_{\Pt{P}_2} \utilde{\Pt{P}_1} = \bigl( x_1 x_2 + \Em y_1 y_2, x_1 y_2 +x_2 y_1 + \Si y_1 y_2 \bigr) , \\
                  \Pt{P}_1 \Psub \Pt{P}_2 & = \MN_{\Pt{P}_2}^{-1} \utilde{\Pt{P}_1} = \bigl( x_1 x_2 + \Si x_1 y_2 - \Em y_1 y_2 , x_2 y_1 - x_1 y_2 \bigr) , 
\end{align}
We use $\Psub$ to denote subtraction. The identity of $\mathcal{P}$ is the point $\Pt{O} = (1, 0)$.

Non-trivial elements of the Tate-Shafarevich group for Pell conics, $\Sha (\mathcal{P} / \Z )$, are described by particular conics, Eqn. \eqref{garrrv} which have rational but no rational integer solutions. If $\mathcal{P}$ is a Pell conic of discriminant $\Delta$ and $Q'$ is a binary quadratic form of discriminant $\Delta$, then an integral solution to Eqn. \eqref{garrrv} provides, through unimodular substitution, an integral point of $\mathcal{P}$, otherwise this kind of solution to $\mathcal{P}(\Z )$ is obstructed.

The relationship between the principal binary quadratic form 
\begin{equation}\label{princform}
Q_0(x,y) : x^2 + \Si x y - \Em y^2 ,
\end{equation}
of fundamental discriminant $\Delta $, and the binary quadratic form $Q'(t,u)$, Eqn. \eqref{intrbinform}, of the same discriminant $B^2 - 4 A' C$ via automorphs is known. If $\Delta > 0$, there are infinitely many pairs of rational integers $(x, y)$ satisfying $Q_0(x, y) = 1$. Any such pair provides an automorph of $Q'(t, u)$. Recall that this means that an integer point $\Pt{P} = (x, y)$ of the Pell conic $\mathcal{P} : Q_0(x, y) = 1$ yields a matrix
\begin{equation}\label{defautomorph} 
M_{\Pt{P}} = \smatr{x- \frac{B - \Si }{2} y}{-C y}{A' y}{x + \frac{B + \Si }{2} y} \in \text{Sl}_2( \Z ), 
\end{equation}
and the replacement $\svect{t}{u} \mapsto M_{\Pt{P}} \svect{t}{u}$ does not affect the coefficients of $Q'$. 
\begin{equation*}
Q' \Bigl( M_{\Pt{P}} \svect{t}{u} \Bigr) = Q_0(x, y) Q'(t, u) = Q'(t, u) .  
\end{equation*}
Let 
\begin{equation}\label{bigN}
\MS_{Q'} = \smatr{\frac{1}{\sqrt{A'}}}{\frac{\Si - B}{2 \sqrt{A'}}}{0}{\sqrt{A'}} . 
\end{equation}
The change of variables $\svect{t}{u} \mapsto \MS_{Q'} \svect{x}{y}$ transforms $Q'(t, u)$ into the principal form $Q_0(x, y)$ , and $\svect{x}{y} \mapsto \MS_{Q'}^{-1} \svect{t}{u}$ reverses this: 
\begin{eqnarray} 
\label{substtwo} Q_0(x, y) & = & Q' \Bigl( \MS_{Q'} \svect{x}{y} \Bigr) , \\
\label{substone}  Q'(t, u) & = & Q_0 \Bigl( \MS_{Q'}^{-1} \svect{t}{u} \Bigr) .
\end{eqnarray} 
If $A'$ is a square, $A' = A^2$, then $\MS_{Q'} \in \text{Sl}_2( \Q )$ and the substitutions of Eqn.s \eqref{substone} and \eqref{substtwo} provide a relationship between rational representations of $n$ by $Q'$ and rational representations of $n$ by the principal form $Q_0$. By transport of the structure of $\mathcal{P}(\Q )$, via the map $(x, y) \mapsto \bigl( \frac{x - \beta y}{A}, A y \bigr)$, the set of rational points of the conic 
\begin{equation}\label{sqrQ} 
Q(t, u) = A^2 t^2 + (2 \beta + \Si ) t u + \frac{Q_0(\beta , 1)}{A^2} u^2 = 1 
\end{equation} 
becomes an abelian group with identity $o = \bigl( \frac{1}{A}, 0 \bigr)$. The middle coefficient $B$ of a binary quadratic form of fundamental discriminant $\Delta = 4 \Em + \Si$ may always be expressed with $B = 2 \beta + \Si$, with $\beta \in \Z$, and since $Q$ is a binary quadratic form of fundamental discriminant $\Delta$, the leading coefficient $A^2$ of $Q$ divides $Q_0(\beta , 1)$. The only specialization in Eqn. \eqref{sqrQ} is that the coefficient of $t^2$ is a square. Let 
\begin{center} 
$\mathfrak{F}^2$ {\em be the set of all binary quadratic forms of discriminant $\Delta$ with leading coefficient a square, Eqn. \eqref{sqrQ}, and $0 \leq \beta < A^2$.} 
\end{center}

We will focus our attention on the binary quadratic forms $Q \in \mathfrak{F}^2$. With respect to a form $Q \in \mathfrak{F}^2$, we define $ L_{Q,q}$, while $\MM_{\Pt{P}}$ and $\MS_{Q'}$, Eqn.s \eqref{defautomorph}, \eqref{bigN} become
\begin{eqnarray}
\label{subtrmp}           \ML_{Q,q} & = & \smatr{A^2 t + (\beta + \Si ) u}{\beta t + \frac{Q_0(\beta , 1)}{A^2} u}{-u}{t} , \\
\label{automorphmat}   \MM_{\Pt{P}} & = & \smatr{x - \beta y}{- \frac{Q_0(\beta , 1)}{A^2} y}{A^2 y}{x + ( \beta + \Si ) y} , \\
\label{Nin}                 \MS_{Q} & = & \smatr{\frac{1}{A}}{\frac{-\beta }{A}}{0}{A} .
\end{eqnarray}

This article is concerned with the correspondence, Eqn.s \eqref{substtwo} and \eqref{substone}, between algebraic integer points $q = (t, u)$ satisfying Eqn. \eqref{sqrQ} ranging over any form $(A^2, B, C)$ of fundamental discriminant $\Delta $, and the group of points of the Pell conic $\mathcal{P}$ over algebraic numbers. This has been motivated by the desire to justify defining the Tate-Shafarevich group for Pell conics as the kernel of a homomorphism between cohomology groups: $H^1(G, \mathcal{P}(\overline{\Z })) \rightarrow  H^1(G, \mathcal{P}(\overline{\Q }))$. Lemmermeyer \cite{LemNewBook} has shown that binary quadratic forms of discriminant $\Delta $ may be considered as torsors $\mathcal{T} : A t^2 + B t u + C u^2 = 1$ for $Q_0(x, y) = 1$, where the addition map for the principal homogeneous space is defined by automorphs of binary quadratic forms. The torsors which have rational but no integer points are the torsors corresponding to non-trivial elements of the Tate-Shafarevich group $\Sha (\mathcal{P} / \Z )$ for Pell conics. Lemmermeyer \cite{LemNewBook} has proved that $\Sha (\mathcal{P} / \Z ) \simeq \text{Cl}^+(K)^2$. Consideration of elliptic curves indicates that $\Sha (\mathcal{P} / \Z )$ may have a cohomological definition. We will give two proofs of the isomorphism $$\text{Cl}^+(K)^2 \simeq \text{ker}\bigl( H^1(G,\mathcal{P}(\overline{\Z})) \rightarrow H^1(G,\mathcal{P}(\overline{\Q})) \bigr) ,$$ and thereby justify the cohomological definition of the Tate-Shafarevich group $\Sha (\mathcal{P} / \Z )$ for Pell conics. The first proof uses the approach of taking Galois cohomology on a short exact sequence of $G_{\overline{\Q}/ \Q}$-modules, while the second proof offers insight on the addition and subtraction maps of a principal homogeneous space but applies a tedious calculation.

Schmid \cite{Schmid} studied the group structure of points of the Pell equation $\mathcal{C} : x^2 - d y^2 = 1$ defined over the ring of integers $\mathcal{O}_K$ of a given number field $K$, with $d \in \mathcal{O}_K$, by looking at the injective map $\mathcal{C}(\mathcal{O}_K) \rightarrow \mathcal{O}_K[\sqrt{d}]$, sending $(x, y)$ to $x + y\sqrt{d}$. If $K = \overline{\Q}$, the algebraic closure of $\Q$, and $d \in \Z$, it is a clear consequence of Schmid's results that $\mathcal{C}(\overline{\Z})$, over all algebraic integers, is not a finitely generated group.

A surjective homomorphism was given in \cite{PellSurf} from the group of primitive integer points of a Pell surface $\mathcal{S}_n : Q_0(B, C) = A^n$, again with $\Delta = \Si +4 \Em$ a fundamental discriminant, onto the $n$-torsion subgroup of the narrow ideal class group $\text{Cl}^+(K)[n]$, of the quadratic number field $K = \Q(\sqrt{\Delta })$. A primitive integer point $(A, B, C) \in \mathcal{S}_n(\Z)$ satisfies 
\begin{equation}\label{primgcdrat}
\text{gcd}(A, C) = 1 = \text{gcd}(A, \Delta ) .
\end{equation} 
When $n = 2$, a point $(A, B, C) \in \mathcal{S}_2(\Z )$ corresponds to a primitive rational point of $\mathcal{P}$. Of course all rational points $\Bigl( \frac{B}{A}, \frac{C}{A} \Bigr) \in \mathcal{P}(\Q )$ may be written to satisfy Eqn. \eqref{primgcdrat}.

Soleng \cite{Soleng} considered `primitive points' of the elliptic curve $E : y^2 = x^3 + a_2 x^2 + a_4 x + a_6$, with $a_2, a_4, a_6 \in \Z$, where the point $\bigl( \frac{A}{C^2}, \frac{B}{C^3} \bigr) \in E(\Q)$ with $\text{gcd}(A, C) = \text{gcd}(B, C) = 1$ belongs to the subgroup $E(\Q)_{\text{prim}}$ if $\text{gcd}(A, 2 B, A^2 + a_2 A C^2 + a_4 C^4) = 1$. Soleng showed that there is a homomorphism mapping $E(\Q)_{\text{prim}}$ to the ideal class group of $\Q(\sqrt{a_6})$.

We would like to define what it means for an algebraic point of a Pell conic to be `primitive'. Once this has been defined we will demonstrate that the primitive algebraic points form a subgroup $\mathcal{P}(\overline{\Q})_{\text{prim}}$ of $\mathcal{P}(\overline{\Q})$, and there is a surjective homomorphism from this subgroup onto the narrow ideal class group of squared ideals $\text{Cl}^+(K)^2$ of a quadratic number field $K = \Q(\sqrt{\Delta })$. In fact we will show that there is an exact sequence 
$$\begin{CD} 
1 @>>> \mathcal{P}(\Q)\oplus \mathcal{P}(\overline{\Z}) @>>> \mathcal{P}(\overline{\Q})_{\text{prim}} @>>> \text{Cl}^+(K)^2 @>>> 1 .
\end{CD}$$
The denominator of the algebraic number $\alpha \in \overline{\Q}$, denoted $\text{den}(\alpha )$, was defined in \cite{ARW} to be the least positive rational integer such that $\text{den}(\alpha ) \alpha $ is an algebraic integer. We use $N_{\overline{\Q}/\Q}(\alpha )$ for the norm of $\alpha \in \overline{\Q}$, the product of the roots of the minimum polynomial of $\alpha$. For the ring of all algebraic integers we use $\overline{\Z}$.
\begin{definition}\label{indepvtwo} 
Define $\mathcal{P}(\overline{\Q})_{\text{prim}}$ to be the set of all $(x, y) \in \mathcal{P}(\overline{\Q})$, letting $A = \text{den}(y)$, such that 
\begin{itemize}
\item $Ax $ is an algebraic integer,
\item $\text{gcd}(A^2 N_{\overline{\Q}/\Q}(y), A^2) = 1 = \text{gcd}(A, \Delta )$,
\item letting $a$ be the least positive integer satisfying $a \equiv (A^2 N_{\overline{\Q}/\Q}(y))^{-1} \pmod{A^2}$, 
$$b = \left\{\begin{array}{ll} 
a A^2 N_{\overline{\Q}/\Q}(y)\frac{x}{y} &\mbox{if $y \not= 0$} \\ 
0 &\mbox{if $y = 0$} 
\end{array}\right. $$ is a rational integer.
\end{itemize}
\end{definition}

\section{$\mathcal{P}(\overline{\Q})_{\text{prim}}$ is a subgroup of $\mathcal{P}(\overline{\Q})$}

We can show that there is an equivalent definition of primitive algebraic point of a Pell conic and we will use this equivalence to show that $\mathcal{P}(\overline{\Q})_{\text{prim.}}$ is a subgroup of $\mathcal{P}(\overline{\Q})$.
\begin{lemma}\label{equaldens}
Let $\alpha \in \overline{\Q}$. Then $\text{den}(\alpha ) = \text{den}(N_{\overline{\Q}/\Q}(\alpha )/\alpha )$.
\end{lemma}

\begin{proof}
In \cite{ARW} we also have the definition of the denominator of a polynomial $f \in \Z[x]$, denoted $\text{den}(f)$, which is the least positive integer for which $\text{den}(f)\alpha $ is an algebraic integer for all zeros $\alpha $ of $f$. We find in \cite{ARW} the identity $\text{den}(\alpha ) = \text{den}(f)$ whenever $f$ is the minimum polynomial of $\alpha $ over $\Z$. Since $\alpha $ and $N_{\overline{\Q}/\Q}(\alpha )/\alpha $ share a minimum polynomial, the result follows.
\end{proof}

Before we prove the equality of two definitions of the primitive subgroup, let us first remark that if $\frac{x - b y}{\text{den}(y)}$ is an algebraic integer then so is $\text{den}(y) x$.

\begin{lemma}\label{equivadef}
Let $(x, y) \in \mathcal{P}(\overline{\Q})$ and $A = \text{den}(y)$. There exists a rational integer $b$ satisfying 
\begin{itemize}
\item $A^2 \mid Q_0(b, 1)$, and
\item $\frac{x - b y}{A}$ is an algebraic integer
\end{itemize}
if and only if $(x, y) \in \mathcal{P}(\overline{\Q})_{\text{prim}}$.
\end{lemma}

\begin{proof} 
First assume that $(x, y) \in \mathcal{P}(\overline{\Q})_{\text{prim}}$. Let $a$ be as in Definition \ref{indepvtwo}, $a_2 = \frac{1}{A^2}\Bigl(1- a A^2 N_{\overline{\Q}/\Q}(y) \Bigr)$, and $b = a A^2 N_{\overline{\Q}/\Q}(y)\frac{x}{y}$ is a rational integer. Then $\frac{x - b y}{A} = a_2 A x \in \overline{\Z}$. Since $A(x+(b + \Si )y) \in \overline{\Z}$, we have the algebraic integer $$A(x + (b + \Si )y) \frac{x - b y}{A}= 1 - Q_0(b, 1) y^2$$ so that $Q_0(b, 1) y^2 \in \overline{\Z}$. It follows that $A^2 \mid Q_0(b, 1)$. Conversely we assume there exists a rational integer satisfying $Q_0(b, 1) \equiv 0 \pmod{A^2}$ and $\frac{x - b y}{A}$ is an algebraic integer. Then clearly $A x \in \overline{\Z}$. Let $\rho = \frac{x - b y}{A}$. Observe that 
\begin{equation}\label{vergon} 
A^2\rho^2 +(2 b + \Si )\rho (A y) +\frac{Q_0(b, 1)}{A^2}(A y)^2 = 1 .
\end{equation} 
Multiplying Eqn. \eqref{vergon} by $A \frac{N_{\overline{\Q}/\Q}(y)}{y} \in \overline{\Z}$, we obtain 
\begin{equation}\label{vergontwo} 
A^2 \rho^2 A \frac{N_{\overline{\Q}/\Q}(y)}{y} +(2 b + \Si )\rho (A^2 N_{\overline{\Q}/\Q}(y)) + A^2 N_{\overline{\Q}/\Q}(y)\frac{Q_0(b,1)}{A^2}(A y) = A \frac{N_{\overline{\Q}/\Q}(y)}{y} .
\end{equation}
Let $p$ be a prime divisor of $\text{gcd}(A^2 \cdot N_{\overline{\Q}/\Q}(y), A^2)$. Then dividing Eqn. \eqref{vergontwo} by $p$, it follows that $\frac{1}{p} A \frac{N_{\overline{\Q}/\Q}(y)}{y} \in \overline{\Z}$, contradicting $A = \text{den}(y) = \text{den}\bigl( N_{\overline{\Q}/\Q}(y)/y \bigr)$. Therefore $\text{gcd}(A^2\cdot N_{\overline{\Q}/\Q}(y), A^2) = 1$. We find that $A^2 \mid Q_0(b,1)$ implies that $\text{gcd}(A,\Delta ) = 1$: Let $p$ be a prime divisor of both $A$ and $\text{gcd}(A,\Delta )$. Then $p^2 \mid (2 b + \Si )^2 - \Delta $, but since $p \mid \Delta $, $p^2 \mid (2 b + \Si )^2$ so that $p^2 \mid \Delta $, a contradiction unless $p = 2$. In this event, we have $\Si = 0$ so that $4 \mid b^2 - \Em$. There are no solutions to the congruence $b^2 \equiv \Em \pmod{4}$ when $\Em \equiv 2$ or $3 \pmod{4}$, which are the only possibilities for $\Delta = 4 \Em$ since $\Delta $ is fundamental. 
\end{proof}

For a primitive algebraic point $\Pt{P} = (x, y) \in \mathcal{P}(\overline{\Q})_{\text{prim}}$, there exist unique integers $A = \text{den}(y)$ and $\beta $, the least non-negative integer satisfying $\beta \equiv b \pmod{A^2}$ where $b$ is as in Definition \ref{indepvtwo}, which we will refer to as the {\em denominator} and {\em ratio} of the point $\Pt{P}$, respectively. We will refer to $\ET = \frac{x - \beta y}{A}$ as the {\em quotient} of the point $\Pt{P}$, where $A$ is the denominator and $\beta $ is the ratio, and $\EU = A y$ as the {\em numerator} of the point $\Pt{P}$. We reserve the symbols $\ET$ and $\EU$ for the quotient and numerator of a point $\Pt{P} \in \mathcal{P}(\overline{\Q})_{\text{prim}}$. From here on, the symbols $B$, $\beta^+$, $\beta_{12}^-$, and $\beta^{\times}$ will always refer to
\begin{eqnarray*}
B & = & 2 \beta + \Si , \\
\beta^+ & = & \beta_1 + \beta_2 + \Si , \\
\beta_{12}^- & = & \beta_1 - \beta_2 , \\
\beta^{\times} & = & \beta_1 \beta_2 + \Em ,
\end{eqnarray*}
since $\beta^+$, $\beta_{12}^-$, and $\beta^{\times}$ will often appear in composition of binary quadratic forms. We will occasionally require a $\beta_3$, in which case $\beta_{13}^-$ and $\beta_{23}^-$ refer to $\beta_1 - \beta_3$, and $\beta_2 - \beta_3$.

Lemma \ref{equivadef} shows that $\mathcal{P}(\overline{\Q})_{\text{prim}}$ is the set of all points $\Pt{P} \in \mathcal{P}(\overline{\Q})$ which came from algebraic integer points $q = (t, u)$ of some arbitrary binary quadratic form of fundamental discriminant $\Delta$, with leading coefficient an integer square, $A^2$ as in Eqn. \eqref{sqrQ}, and with $A^2 > b \geq 0$, under the map
\begin{equation} 
\varphi^{-1} : \svect{t}{u} \mapsto \MS_Q^{-1} \svect{t}{u} = \svect{x}{y} ,
\end{equation}
with $N_Q$ as in Eqn. \eqref{Nin}. The definition of $\mathcal{P}(\overline{\Q})_{\text{prim}}$ indicates, given $\Pt{P} \in \mathcal{P}(\overline{\Q})_{\text{prim}}$, which binary quadratic form the point $\Pt{P}$ came from.

The quotient $\ET$ is equal to the $t$-coordinate of $\varphi ( \Pt{P} )$ where $\Pt{P} = (x, y) \in \mathcal{P}(\overline{\Q})_{\text{prim}}$, 
\begin{equation}\label{defvarph} 
\varphi : \svect{x}{y} \mapsto \MS_Q \svect{x}{y} = \svect{t}{u} ,
\end{equation}
and the $A$ and $\beta$ appearing in the matrix $N_Q$ are the denominator and ratio of the point $\Pt{P}$.

Observe that if $x, y, z$ are positive integers, the greatest common divisor satisfies 
\begin{eqnarray}
\label{firstgcd} \text{gcd}(x, y)^2 & = & \text{gcd}(x^2, y^2) , \\
\label{secondgcd} \text{gcd}(x, y, z)^2 & = & \text{gcd}(x^2, y^2, z^2) .
\end{eqnarray}

\begin{lemma}\label{comparegcds} 
Let $\Pt{P}_1, \Pt{P}_2 \in \mathcal{P}(\overline{\Q})_{\text{prim}}$ with respective denominators $A_1, A_2$ and ratios $\beta_1, \beta_2$.
Let $\hat{e} = \text{gcd}(A_1^2, A_2^2, \beta_1^+ )$ and $e = \text{gcd}(A_1, A_2, \beta^+ )$. Then $\hat{e} = e^2$.
\end{lemma}

\begin{proof}
By Lemma \ref{equivadef}, for $j = 1, 2$, $A_j^2 \mid Q_0(\beta_j, 1)$, and since $e^2 \mid A_1^2$ and $e^2 \mid A_2^2$, 
\begin{equation*}
\beta^+ \beta_{12}^- = Q_0(\beta_1, 1)-Q_0(\beta_2, 1) \equiv 0 \pmod{e^2} .
\end{equation*}
Therefore $e^2 \mid \beta^+ \beta_{12}^-$. To show that $\text{gcd}(\beta_{12}^-, e) = 1$, let $p$ be a prime divisor of $\beta_{12}^-$ and $e$. Then $\beta^+ \equiv 0 \pmod{p}$ and $\beta_{12}^- \equiv 0 \pmod{p}$, which implies that $B_1 \equiv 0 \pmod{p}$, but since $B_1^2 -\Delta \equiv 0 \pmod{A_1^2}$, we have $p \mid \Delta $ contradicting $\text{gcd}(\Delta , A_1) = 1$. Therefore $e^2 \mid \beta^+$. It follows that $e^2 \mid \hat{e}$. By Eqn. \eqref{secondgcd}, we have $\text{gcd}(A_1^2, A_2^2, (\beta^+)^2) = e^2$. Since $\hat{e} \mid \text{gcd}(A_1^2, A_2^2, (\beta^+)^2) = e^2$, the result follows.
\end{proof}

Again let $\omega = \frac{\Si +\sqrt{\Delta }}{2}$. For each primitive point $\Pt{P} \in \mathcal{P}(\overline{\Q})_{\text{prim}}$, with denominator $A$ and ratio $\beta $, we may attach ideals $\mathfrak{a}'_{\Pt{P}}$ and $\mathfrak{a}_{\Pt{P}}$ of the ring of integers $\Z[ \omega ]$ of $\Q(\sqrt{\Delta })$, and binary quadratic forms $Q'_{\Pt{P}}$ and $Q_{\Pt{P}}$ of discriminant $\Delta $ as follows 
\begin{eqnarray}
\label{attachone} \mathfrak{a}'_{\Pt{P}} & = & (A, \beta +\omega ) , \\
\label{attachtwo}  \mathfrak{a}_{\Pt{P}} & = & (A^2, \beta +\omega ) , \\
\label{attachthree}          Q'_{\Pt{P}} & = & \bigl( A, 2 \beta +\Si , Q_0(\beta , 1)/A \bigr) , \\
\label{attachfour}            Q_{\Pt{P}} & = & \bigl( A^2, 2 \beta +\Si , Q_0(\beta , 1)/A^2 \bigr) .
\end{eqnarray}

If $\Pt{P}_1, \Pt{P}_2, \Pt{P}_1 \Padd \Pt{P}_2 \in \mathcal{P}(\overline{\Q})_{\text{prim}}$ with associated forms $Q_{\Pt{P}_1}$, $Q_{\Pt{P}_2}$, and $Q_{\Pt{P}_1 \Padd \Pt{P}_2}$ then the latter form is equivalent under unimodular substitution to the composed form $Q_{\Pt{P}_1} \cdot Q_{\Pt{P}_2}$. We use composition of forms to show that $\mathcal{P}(\overline{\Q})_{\text{prim}}$ is closed under $\Padd$, the group law for Pell conics. We point out that $\varphi (\Pt{P}) = (t, u)$, with $\varphi$ as in Eqn. \eqref{defvarph}, is an algebraic integer point of the conic $Q_{\Pt{P}}(t, u) = 1$.  

The following theorem expresses, as will be demonstrated in Section 4.2, that the set of all algebraic integer points of all conics $Q(t, u) = 1$ where $Q$ is a binary quadratic form of fundamental discriminant $\Delta$, has leading coefficient an integer square $A^2$, Eqn. \eqref{sqrQ}, and $A^2 > \beta \geq 0$, forms a group with identity $\Pt{O} = (1, 0)$ and is isomorphic to a subgroup of $\mathcal{P}(\overline{\Q})$.

\begin{theorem}\label{thelemm} 
$\mathcal{P}(\overline{\Q})_{\text{prim}}$ is a subgroup of $\mathcal{P}(\overline{\Q})$ with respect to the group law \eqref{grplawnm} for Pell conics. 
\end{theorem}
\begin{proof} 
We will demonstrate: Let $\Pt{P}_1 = (x_1, y_1), \Pt{P}_2 = (x_2, y_2) \in \mathcal{P}(\overline{\Q})_{\text{prim}}$. Let $A_1$ and $A_2$ be the denominators of the points $\Pt{P}_1$ and $\Pt{P}_2$ respectively and let $\beta_1$ and $\beta_2$ be the ratios of $\Pt{P}_1$ and $\Pt{P}_2$ respectively. Let $\Pt{P}_3 = (x_3,y_3) = \Pt{P}_1 \Padd \Pt{P}_2 $, $e^2 = \text{gcd}(A_1^2, A_2^2, \beta^+ )$, $A_3 = \frac{A_1 A_2}{e^2}$, $j, k, \ell$ be any rational integers satisfying $A_1^2 j + A_2^2 k + \beta^+ \ell = e^2$, and $b_3 = \frac{A_1^2}{e^2}\beta_2 j+ \frac{A_2^2}{e^2}\beta_1 k + \frac{\beta^{\times}}{e^2} \ell$. Then, letting $\beta_3$ be the least non-negative integer satisfying $\beta_3 \equiv b_3 \pmod{A^2}$, $\frac{x_3 - \beta_3 y_3}{A_3}$ is an algebraic integer, $A_3 = \text{den}(y_3)$, and $A_3^2 \mid Q_0(b_3, 1)$. In other words we use the coefficients $A_3$ and $2 \beta_3 + \Si $ of a binary quadratic form in the class of the composed form $Q_{\Pt{P}_1} \cdot Q_{\Pt{P}_2}$ to provide a denominator and ratio for the point $\Pt{P}_3 = \Pt{P}_1 \Padd \Pt{P}_2$. 

Let $\ET_1$, and $\ET_2$ be the quotients, $\EU_1$ and $\EU_2$ be the numerators of $\Pt{P}_1$ and $\Pt{P}_2$ respectively. Let $\gamma_1 = \frac{Q_0(\beta_1, 1)}{A_1^2}$ and $\gamma_2 = \frac{Q_0(\beta_2, 1)}{A_2^2}$. The identities
\footnotesize
\begin{eqnarray*} 
                  A_3 y_3 & = & \frac{A_1^2}{e^2} \ET_1 \EU_2 + \frac{A_2^2}{e^2} \ET_2 \EU_1 + \frac{\beta^+}{e^2} \EU_1 \EU_2 , \\ 
\frac{x_3 - b_3 y_3}{A_3} & = & A_1 x_1 \ET_2 j + A_2 x_2 \ET_1 k + \beta^+ \ET_1 \ET_2 \ell + \ET_2 \EU_1 ( \gamma_1 \ell - \beta_2 j ) + \ET_1 \EU_2 ( \gamma_2 \ell - \beta_1 k ) - \EU_1 \EU_2 ( \gamma_2 j + \gamma_1 k ) , \\
              Q_0(b_3, 1) & = & A_3^2 \Bigl( A_1^2 \gamma_2 j^2 + (2 \beta^{\times} + \Si \beta^+ -\Delta ) j k + B_1 \gamma_2 j \ell + A_2^2 \gamma_1 k^2 + B_2 \gamma_1 k \ell +\gamma_1 \gamma_2 \ell^2 \Bigr) ,
\end{eqnarray*}
\normalsize
show that $A_3 y_3$ and $\frac{x_3 - b_3 y_3}{A_3}$ are algebraic integers, $\text{den}(y_3) \mid A_3$, and $A_3^2 \mid Q_0(\beta_3,1)$. Now observe that
$$\frac{x_3 - \beta_3 y_3}{A_3} \cdot \text{den}(y_3) \cdot (x_3 + (\beta_3 + \Si ) y_3) = \frac{\text{den}(y_3)}{A_3} \bigl( 1 - Q_0(\beta_3, 1)y_3^2 \bigr) $$ 
If $p$ is a prime divisor of $\frac{A_3}{\text{den}(y_3)}$ then we have the algebraic integer
\begin{eqnarray*} 
\frac{1}{p}\bigl( 1-Q_0(\beta_3, 1)y_3^2\bigr) & = & \frac{1}{p}\bigl( 1 - r p^2 \bigl( \text{den}(y_3) \bigr)^2 y_3^2 \bigr) \text{ where } r \in \Z , \\ 
                                               & = & \frac{1}{p} - r p \bigl( \text{den}(y_3)\bigr)^2 y_3^2 . 
\end{eqnarray*} 
Since $rp\bigl( \text{den}(y_3)\bigr)^2 y_3^2 \in \overline{\Z}$, $\frac{1}{p} \in \overline{\Z}$, a contradiction. Therefore $A_3 \mid \text{den}(y_3)$ and we must have $A_3 = \text{den}(y_3)$. 

We have shown that $A_3$ and $\beta_3$ are the denominator and ratio of $\Pt{P}_3$. Therefore $\mathcal{P}(\overline{\Q})_{\text{prim}}$ is closed under $\Padd$. By definition, $(1,0) \in \mathcal{P}(\overline{\Q})_{\text{prim}}$. Clearly $\mathcal{P}(\overline{\Q})_{\text{prim}} \subseteq \mathcal{P}(\overline{\Q})$. It remains to check that if $\Pt{P} \in \mathcal{P}(\overline{\Q})_{\text{prim}}$ then the inverse $-\Pt{P}$ also belongs to $\mathcal{P}(\overline{\Q})_{\text{prim}}$. Let $\Pt{P} = (x,y) \in \mathcal{P}(\overline{\Q})_{\text{prim}}$ with $A = \text{den}(y)$. Clearly $A = \text{den}(-y)$. If $\beta $ is the ratio of $\Pt{P}$, it is easy to verify that $A^2 - \beta - \Si $ is the ratio of $-\Pt{P}$.
\end{proof}

\begin{example} 
Consider the Pell conic $\mathcal{P} : x^2 + x y - 57 y^2 = 1$ and $\Pt{P} = \bigl( \frac{\sqrt{-1}}{5}, \frac{-2\sqrt{-1}}{15} \bigr) \in \mathcal{P}(\overline{\Q})$. Now $A = \text{den}(y) = 15$, $A^2 N_{\overline{\Q}/\Q}(y) = 4$ and $a \equiv 4^{-1} \equiv 169 \pmod{225}$, so $b = -1014$, a rational integer. The ratio of $\Pt{P}$ is $\beta = 111$. We may consider $\Pt{P}$ to have come from an algebraic integer point of $225 t^2 + 223 t u + 55 u^2 = 1$. This point is $\smatr{\frac{1}{15}}{\frac{-111}{15}}{0}{15} \svect{\frac{\sqrt{-1}}{5}}{\frac{-2 \sqrt{-1}}{15}} = \svect{\sqrt{-1}}{-2 \sqrt{-1}}$. 
\end{example}

\subsection*{Surjective homomorphisms}

The maps sending the point $\Pt{P}$ to classes of the ideals and forms given by Eqn.s \eqref{attachone}, \eqref{attachtwo}, \eqref{attachthree}, and \eqref{attachfour} are homomorphisms. Reflection on the following theorem of Birch \cite{Birch} will be helpful in showing that some of these maps are surjective.

\begin{theorem}[Birch] \label{birch}
Let $K$ be a number field with integers $\mathcal{O}_K$. If the polynomial $f(T, U) = \sum a_{ij}T^iU^j$, with $a_{ij} \in \mathcal{O}_K$, is homogeneous and the $a_{ij}$ are coprime, then there exist units $t, u \in \mathcal{O}_L$ satisfying $f(t, u) = 1$ where $L$ is an extension of $K$.
\end{theorem}

Specifically we will use the following corollary.

\begin{corollary}\label{birchcor}
Let $Q = (A, B, C)$ be a primitive binary quadratic form of fundamental discriminant $\Delta$. There exist units $t, u$ of the ring of integers $\mathcal{O}_L$ of some algebraic extension $L$ of $\Q$ such that $Q(t,u) = 1$.
\end{corollary} 

For each binary quadratic form $Q = (A, B, C)$ of fundamental discriminant $\Delta = \Si + 4 \Em$ there exists a unique integer $\beta $ such that $Q = (A, 2 \beta + \Si , Q_0( \beta , 1) / A )$. The corresponding ideal of the ring of integers $\mathcal{O}_K$ of the field $K = \Q(\sqrt{\Delta })$, under the known isomorphism between classes of forms and narrow ideal classes, is $\mathfrak{a} = ( A , \beta +\omega )$ where $\{ A, \beta + \omega \}$ is an integral basis for $\mathfrak{a}$ and $\omega = \frac{\Si +\sqrt{\Delta }}{2}$. We refer to Lenstra \cite{Lenstramart} for the following compositions and group $\mathcal{F}$.  

Let $\mathcal{F} $ be the set of all classes $[Q]_\mathcal{F}$ of binary quadratic forms $Q$ of discriminant $\Delta $ equivalent under unimodular substitution with $\smatr{1}{\ell }{0}{1}$ and $\ell \in \Z$. The set of all such matrices forms a subgroup of $\text{SL}_2(\Z )$ with matrix multiplication and is isomorphic to the additive group $\Z$. The usual narrow class group of forms places two forms in the same class if they are equivalent under unimodular substitution with a general element of $\text{SL}_2(\Z)$. It is convenient to use $\mathcal{F}$ since we will later define a group based on individual forms in $\mathcal{F}$.
\begin{remark}
Letting $B_1 = 2 \beta_1 +\Si $ and $B_2 = 2 \beta_2 + \Si $, the forms $Q_1 = (A_1, B_1, C_1)$, and $Q_2 = (A_2, B_2, C_2)$ are in the same class of $\mathcal{F}$ if and only if 
\begin{eqnarray*}
    \label{nd} A_1 & = & A_2 , \\
\label{sx} \beta_1 & \equiv & \beta_2 \pmod{A_1} .
\end{eqnarray*}
\end{remark}

The set $\mathcal{F} $ is a group with respect to the binary operation 
\begin{equation}\label{gkkn} 
[Q_1]_{\mathcal{F} }\cdot [Q_2]_{\mathcal{F} } = [Q_1 \cdot Q_2]_{\mathcal{F} } = [Q_3]_{\mathcal{F} } ,
\end{equation} 
where $Q_3 = ( A_3, 2 \beta_3 + \Si , Q_0(\beta_3, 1)/A_3 )$, and 
\begin{eqnarray}
\label{gcde}            e & = & \text{gcd}(A_1, A_2, \beta^+ ), \\
\label{Athree}        A_3 & = & (A_1 A_2)/e^2, \\
\label{betathree} \beta_3 & = & \frac{A_1}{e} \beta_2 j + \frac{A_2}{e} \beta_1 k + \frac{\beta^{\times}}{e} \ell
\end{eqnarray}
and $j, k, \ell$ are any rational integers satisfying
\begin{equation}\label{tuv}
A_1 j + A_2 k + \beta^+ \ell = e .
\end{equation}

If $Q = (A, 2 \beta + \Si , Q_0(\beta , 1)/A )$ and we wish to determine $[Q]_{\mathcal{F} }\cdot [Q]_{\mathcal{F} }$, we have $A_1 = A_2 = A$, $\beta_1 = \beta_2$ and Eqn. \eqref{gcde} becomes $e = \text{gcd}(A, A, 2 \beta + \Si ) = 1$, Eqn. \eqref{Athree} becomes $A_3 = A^2$, and letting $w = j + k$, Eqn. \eqref{tuv} becomes $A w + (2 \beta + \Si ) \ell = 1$ so that
\begin{eqnarray*} 
\beta_3 & = & A \beta w + ( \beta^2 + \Em )\ell , \\ 
        & = & \beta - \beta (2 \beta + \Si )\ell + (\beta^2 + \Em ) \ell , \\ 
        & = & \beta - \ell Q_0( \beta , 1) , 
\end{eqnarray*} 
and
\begin{equation*}
[Q]_{\mathcal{F}} \cdot  [Q]_{\mathcal{F}} = \Bigl[ \Bigl( A^2, 2\beta + \Si - 2 \ell A \frac{Q_0(\beta , 1)}{A} ,w \frac{Q_0(\beta , 1)}{A} +\ell^2 \frac{Q_0(\beta , 1)^2}{A^2} \Bigr) \Bigr]_{\mathcal{F}} .
\end{equation*}
If $A^2 \mid Q_0(\beta ,1)$ then $ [Q]_{\mathcal{F}} \cdot  [Q]_{\mathcal{F}} =  [(A^2, 2 \beta + \Si , Q_0(\beta ,1)/A^2 )]_{\mathcal{F}} $.
\begin{definition}\label{Fsquared} 
Define $\mathcal{F}^2$ to be the subgroup of $\mathcal{F}$ consisting of classes of forms in $[(A^2, 2 \beta + \Si ,Q_0(\beta ,1)/A^2 )]_{\mathcal{F}}$ where $A$ and $\beta $ are any integers satisfying $A^2 \mid Q_0(\beta , 1)$.
\end{definition}

Before moving on it is important to illustrate the group law on the infinite class group $\mathcal{F}^2$. Assume we wish to compose the binary quadratic forms 
\begin{eqnarray*}
Q_1 & = & (A_1^2, 2 \beta_1 + \Si , Q_0(\beta_1, 1) / A_1^2 ) , \\
Q_2 & = & (A_2^2, 2 \beta_2 + \Si , Q_0(\beta_2, 1) / A_2^2 )
\end{eqnarray*}
in $\mathcal{F}^2$. To do so, we first compute $$\hat{e} = \text{gcd}(A_1^2, A_2^2, \beta^+ ) = e^2 = \text{gcd}(A_1 , A_2, \beta^+ )^2$$ and the Bezout numbers $j, k, \ell$ satisfying $A_1^2 j + A_2^2 k + \beta^+ \ell = e^2$. Next compute $A_3^2 = A_1^2 A_2^2 / e^4$. Finally, $$\beta_3 = \frac{A_1^2}{e^2} \beta_2 j + \frac{A_2^2}{e^2} \beta_1 k + \frac{\beta^{\times}}{e^2} \ell ,$$ which may be reduced modulo $A_3^2$. A representative of the class of $Q_1 \cdot Q_2$ in $\mathcal{F}^2$ is the binary quadratic form $$Q_3 = (A_3^2, 2 \beta_3 + \Si , Q_0(\beta_3, 1) / A_3^2 ) .$$

The inverse of $[(A^2, 2 \beta + \Si , \gamma )]_{\mathcal{F}}$ is $[(A^2, - 2 \beta - \Si , \gamma )]_{\mathcal{F}}$, and the identity is $[Q_0]_{\mathcal{F}}$. 

As usual let $\text{Cl}^+(\Delta )^2$ denote the subgroup of the narrow class group of binary quadratic forms $Q$ of fundamental discriminant $\Delta$ consisting of classes $[Q']_{\Delta } \cdot [Q']_{\Delta }$ where $[Q']_{\Delta } \in \text{Cl}^+(\Delta )$.

\begin{theorem}\label{homneeded} 
Let $A$ be the denominator and $\beta$ the ratio of the point $\Pt{P} \in \mathcal{P}(\overline{\Q})_{\text{prim}}$. The map given by $\Pt{P} \mapsto [Q'_{\Pt{P}}]_{\mathcal{F}}$ is a homomorphism. There is a commutative diagram of exact sequences of groups
$$ \begin{CD}  
      @.       1 @.        1 @.    1       \\    
      @. @VVV    @VVV     @VVV     \\        
    1 @>>> \mathcal{P}(\overline{\Z }) @>>> \mathcal{P}(\Q) \oplus \mathcal{P}(\overline{\Z }) @>>> \mathcal{P}(\Q) @>>> 1 \\     
      @. @VVV    @VVV     @VVV     \\    
    1 @>>> \mathcal{P}(\overline{\Z }) @>>> \mathcal{P}(\overline{\Q })_{\text{prim}}  @>{\theta }>> \mathcal{F}^2 @>>> 1 \\
      @. @VVV    @V{\lambda_K}VV          @V{\phi}VV     \\     
    1 @>>> 1 @>>> \text{Cl}^+(K)^2 @>{\pi }>> \text{Cl}^+(\Delta )^2 @>>> 1 \\ 
      @. @VVV    @VVV          @VVV     \\    
      @.  1 @.       1 @.    1 \\
\end{CD}$$ 
where the maps are given by $\theta : \Pt{P} \mapsto [Q_{\Pt{P}}]_{\mathcal{F}}$, $\phi : [Q]_{\mathcal{F}} \mapsto [Q]_{\Delta }$, $\lambda_K : (x,y) \mapsto [\mathfrak{a}^2]_{K}  \hspace{0.5cm} \mathfrak{a} = (A, \beta + \omega )$, $\pi : [(A, \beta + \omega )^2]_{K} \mapsto [(A^2, 2 \beta + \Si , \gamma )]_{\Delta }$.
\end{theorem}

\begin{proof}
Clearly $\theta (\Pt{P}) \in \mathcal{F}^2$ and the form $Q_{\Pt{P}}$ is uniquely determined by $\Pt{P}$ so $\theta$ is well defined. The proof that $\theta$ is a homomorphism follows directly from the proof that $\mathcal{P}(\overline{\Q})_{\text{prim.}}$ is closed under $\Padd$. We exhibited $A_3$, $\beta_3$, respectively the denominator and ratio of the point $\Pt{P}_3 = \Pt{P}_1 \Padd \Pt{P}_2$, using composition of forms. It will follow that $\theta$ is a homomorphism.  
\begin{tabbing} 
\hspace{2.0cm} $\theta (\Pt{P}_1) \cdot \theta (\Pt{P}_2)$ \= $= [Q_{\Pt{P}_1}]_{\mathcal{F}} \cdot [Q_{\Pt{P}_2}]_{\mathcal{F}}$ \\ \> $= \bigl[ Q_{\Pt{P}_3} \bigr]_{\mathcal{F}}$ \hspace{1.0cm} by Eqn. \eqref{gkkn} and proof of Lemma \ref{thelemm}. \\ \> $ = \theta ( \Pt{P}_3 )$ 
\end{tabbing} 
To show that $\theta$ is surjective let $[Q]_{\mathcal{F}} = [(A^2,2\beta +\Si , \gamma )] \in \mathcal{F}^2$ and let $t,u$ be algebraic integers satisfying $Q(t,u) = 1$, which exist by Corollary \ref{birchcor}. Then $\MS_Q^{-1} \svect{t}{u} \in \mathcal{P}(\overline{\Q})_{\text{prim}}$ and $\theta : \MS_Q^{-1} \svect{t}{u} \mapsto [Q]_{\mathcal{F}}$.

The map $\pi$ is known to be an isomorphism, $\phi $ is clearly a surjective homomorphism and it is simple to see that $\phi \theta = \pi \lambda_K$ so that $\lambda_K$ is a surjective homomorphism. 

We now consider the kernels of the maps $\theta$, $\lambda_K$, $\phi$, and $\pi$. If $\Pt{P}_1 = (x_1, y_1) \in \mathcal{P}(\Q)$ with denominator $A_1$ and quotient $\ET_1$, then $Q_{\Pt{P}_1}(\ET_1 , \EU_1 ) = 1$, where $\ET_1 , A_1 y_1 \in \Z$, meaning that $Q_{\Pt{P}_1} \in [Q_0]_{\Delta }$. If $\Pt{P}_2 \in \mathcal{P}(\overline{\Z})$ then the denominator of $\Pt{P}_2$ is equal to $1$ so that $Q_{\Pt{P}_2} \in [Q_0]_{\Delta }$. Since $\pi \lambda_K$ is a homomorphism, $\pi \lambda_K (\Pt{P}_1 \Padd \Pt{P}_2) \in [Q_0]_{\Delta }$. It follows that $\mathcal{P}(\Q ) \oplus \mathcal{P}(\overline{\Z}) \subseteq \text{ker } \pi \lambda_K $. Conversely assume that $\Pt{P} = (x,y) \in \text{ker } \pi \lambda_K $. Let $A$ be the denominator and $\beta $ the ratio of $\Pt{P}$. Since $\Pt{P} \in \text{ker } \pi \lambda_K $, there exist rational integers $t, u$ satisfying $Q_{\Pt{P}}(t, u) = 1$, from which we obtain the rational point $\Bigl( \frac{A^2 t+\beta u}{A}, \frac{u}{A} \Bigr) = \bigl( \frac{B}{A}, \frac{C}{A} \bigr) \in \mathcal{P}(\Q)$ with denominator $A$, ratio $\beta $ and inverse $\bigl( \frac{B + \Si C}{A}, \frac{-C}{A} \bigr)$. Observe that $\frac{x - \beta y}{A} \in \overline{\Z}$ implies that $\frac{(\beta + \Si ) x - \Em y}{A} \in \overline{\Z}$ so that by multiplying by $C$, we see that $\frac{( B + \Si C ) x + \Em y(-C)}{A} \in \overline{\Z}$. Similarly multiplying $\frac{x - \beta y}{A}$ by $C$ shows that $\frac{x (-C) + (B + \Si C) y + \Si y (-C)}{A} \in \overline{\Z}$. Therefore we have shown that $\Pt{P} - \bigl( \frac{B}{A}, \frac{C}{A} \bigr) \in \mathcal{P}(\overline{\Z})$ so that $\text{ker } \pi \lambda_K \subseteq \mathcal{P}(\Q)\oplus \mathcal{P}(\overline{\Z})$ and the result follows since $\text{ker } \pi \lambda_K = \text{ker } \lambda_K$.
$$\text{ker }\theta = \{ \Pt{P} \in \mathcal{P}(\overline{\Q})_{\text{prim}} \text{ : } [Q_{\Pt{P}}]_{\mathcal{F}} = [Q_0]_{\mathcal{F}} \} ,$$ where $A$ and $\beta $ are respectively the denominator and ratio of the point $\Pt{P}$. Clearly the only possibility is that $A = 1$ and $\beta = 0$. It follows that $(x, y) \in \mathcal{P}(\overline{\Z})$. The kernel-cokernel exact sequence gives the exact sequence 
$$\begin{CD} 
1 @>>> \mathcal{P}(\overline{\Z}) @>>> \mathcal{P}(\Q)\oplus \mathcal{P}(\overline{\Z}) @>>> \text{ker }\phi @>>> 1 
\end{CD}$$
since $\theta $ is surjective. The only possibility is that $\text{ker }\phi \simeq \mathcal{P}(\Q)$.
\end{proof}

The map given by $\Pt{P} \mapsto [Q'_{\Pt{P}}]_{\mathcal{F}}$ is not surjective in general.

\section{Bilinear transformations}

Each class of binary quadratic forms $[(A^2, 2 b + \Si , Q_0(b,1)/A^2)]_{\mathcal{F}}$ of $\mathcal{F}^2$ has a unique representative form $(A^2, 2\beta +\Si , Q_0(\beta , 1)/A^2 ) $, where $\beta $ is the least non-negative integer satisfying $\beta \equiv b \pmod{A^2}$. We will use $\mathfrak{F}^2$ to refer to this set of representative forms of $\mathcal{F}^2$ and note that since there is a bijection between $\mathcal{F}^2$ and $\mathfrak{F}^2$, we also have the same group structure on the set $\mathfrak{F}^2$. Now we consider the collection of all algebraic integer points of all conics $Q(t, u) =1$ where $Q \in \mathfrak{F}^2$, and observe that Gauss' bilinear transformation is a group operation on this collection of points. We define 
\begin{equation*} 
\mathfrak{S} = \bigcup_{Q \in \mathfrak{F}^2} \Bigl\{ (t,u) \in \overline{\Z} \times \overline{\Z} \mid Q(t, u) = 1 \Bigr\} .
\end{equation*} 

\begin{theorem}\label{mainress} 
Let $\Pt{P} = (x,y) \in \mathcal{P}(\overline{\Q})_{\text{prim}}$ with denominator $A$, ratio $\beta$, quotient $\ET$, and numerator $\EU$. Define a map $\varphi : \mathcal{P}(\overline{\Q})_{\text{prim}} \rightarrow \mathfrak{S}$ by $\varphi : \Pt{P} \mapsto ( \ET , \EU ) \text{ satisfying } Q_{\Pt{P}}( \ET , \EU ) = 1$, Eqn. \eqref{defvarph}. Then $\varphi$ is bijective and $\mathfrak{S}$ becomes an abelian group by transport of structure. 
\end{theorem}

\begin{proof} 
If $t, u \in \overline{\Z}$ satisfy $Q(t, u) = 1$ with $Q$ as in Eqn. \eqref{sqrQ}, then $\MS_Q^{-1} \svect{t}{u} \in \mathcal{P}(\overline{\Q})_{\text{prim}}$ and $\varphi \Bigl( \MS_Q^{-1} \svect{t}{u} \Bigr) = (t, u)$. This shows that $\varphi$ is surjective. Assume that $\Pt{P}_1 = (x_1, y_1), \Pt{P}_2 = (x_2, y_2) \in \mathcal{P}(\overline{\Q})_{\text{prim}}$ and $\varphi (\Pt{P}_1) = \varphi (\Pt{P}_2)$. Then $\text{den}(y_1) = \text{den}(y_2) = A$ and $\beta_1 = \beta_2$ so $\MS_Q \utilde{\Pt{P}_1} = \MS_Q \utilde{\Pt{P}_2}$ and thus $\Pt{P}_1 = \Pt{P}_2$. It follows that $\varphi$ is injective.
\end{proof}

\subsection*{The group law on $\mathfrak{S}$}

Let $t_1,u_1,t_2,u_2$ be algebraic integers satisfying $Q_1(t_1,u_1) = 1 = Q_2(t_2,u_2)$, where $Q_1, Q_2 \in \mathfrak{F}^2$ so that $(t_1, u_1), (t_2, u_2) \in \mathfrak{S}$. We define a binary operation $\circ : \mathfrak{S} \times \mathfrak{S} \rightarrow \mathfrak{S}$ by 
\begin{equation}\label{grpcirc}
\varphi (\Pt{P}_1) \circ \varphi (\Pt{P}_2) = \varphi (\Pt{P}_1 \Padd \Pt{P}_2) ,
\end{equation} 
where $\varphi (\Pt{P}_1) = (t_1, u_1)$ and $\varphi (\Pt{P}_2) = (t_2, u_2)$. Letting $\beta_1, \beta_2$ and $\beta_3$ be the ratios of $\Pt{P}_1$, $\Pt{P}_2$ and $\Pt{P}_1 \Padd \Pt{P}_2$ respectively and
\begin{eqnarray*}
\mathcal{A} & = & [a, b, c, d, e^2, f, g, h] , \\
            & = & \Bigl[ 0, \frac{A_1^2}{e^2} , \frac{A_2^2}{e^2}, \frac{\beta^{+}}{e^2}, \text{gcd }(A_1^2, A_2^2, \beta^+ ), \frac{e^2}{A_2^2} \beta_{23}^- , \frac{e^2}{A_1^2} \beta_{13}^- , \frac{e^2}{A_1^2 A_2^2} \bigl( \beta^{\times} - \beta_3 \beta^+ \bigr) \Bigr] .
\end{eqnarray*}
It follows from Eqn. \eqref{grpcirc} that 
\footnotesize
\begin{eqnarray*} 
(t_1, u_1) \circ  (t_2, u_2) & = & \varphi \Bigl( \Bigl( \frac{A_1^2 t_1 + \beta_1 u_1}{A_1}, \frac{u_1}{A_1} \Bigr) \Padd \Bigl( \frac{A_2^2 t_2 + \beta_2 u_2}{A_2}, \frac{u_2}{A_2} \Bigr) \Bigr) , \\ 
                           & = & \varphi \Bigl( \frac{A_1^2 A_2^2 t_1 t_2 + A_1^2 \beta_2 t_1 u_2 + A_2^2 \beta_1 t_2 u_1 + \beta^{\times} u_1 u_2 }{A_1 A_2} , \frac{A_1^2 t_1 u_2 + A_2^2 t_2 u_1 + \beta^+ u_1 u_2}{A_1 A_2} \Bigr) , \\ 
                           & = & \bigl( e^2 t_1 t_2 + f t_1 u_2 + g t_2 u_1 + h u_1 u_2 , b t_1 u_2 + c t_2 u_1 + d u_1 u_2 \bigr) , \\ 
                           & = & (t_3, u_3) , 
\end{eqnarray*}
\normalsize 
where $Q_3 = (A_3^2, 2 \beta_3 + \Si , Q_0(\beta_3, 1)/A_3^2 ) $ is the composed form $Q_1 \cdot Q_2$ under Gauss composition\footnote{The author first started with $\mathfrak{S}$, inspired by Lemmermeyer's \cite{LemPep, LemNewBook} description of Gauss' \cite{Gauss} method of composing forms, known as bilinear transformation, and the relationship to Bhargava cubes \cite{Bhargava}, and discovered $\mathcal{P}(\overline{\Q})_{\text{prim}}$ via the inverse map $\varphi^{-1}$. }, $$(t_3, u_3) = \bigl( e^2 t_1 t_2 + f t_1 u_2 + g t_2 u_1 + h u_1 u_2, b t_1 u_2 + c t_2 u_1 + d u_1 u_2 \bigr)$$ is a bilinear transformation where $$Q_3(t_3, u_3) = Q_1(t_1, u_1) Q_2(t_2, u_2) .$$ It is no coincidence that we use letters $b, c, d, e, f, g, h$ in the description of the group law on $\mathfrak{S}$. There is a Bhargava cube involved. See \cite{Bhargava, LemPep, LemNewBook} for composition of forms via Bhargava cubes.

\section{Galois cohomology of Pell conics}

Recall \cite{Lem03,LemDes3,LemNewBook} that the $2$-torsion subgroup of the Tate-Shafarevich group for Pell conics may be defined $$\Sha(\mathcal{P}/\Z )[2] = \text{ker}\bigl( H^1(G, \mathcal{P}(\overline{\Z}))[2] \rightarrow H^1(G, \mathcal{P}(\overline{\Q}))[2] \bigr) ,$$ and in \cite{LemDes3,LemNewBook} an equivalent definition of $\Sha(\mathcal{P}/\Z)$ without cohomology groups was proved to be isomorphic to the $2$-torsion subgroup of the narrow class group of squared ideal classes of the quadratic number field $K$, $\text{Cl}^+(K)^2[2]$.

See \cite{Milne, SerreGC, Silverman} for introductions to Galois cohomology. Let $G = \text{Gal}(\overline{\Q}/\Q)$ be the Galois group of extension $\overline{\Q} / \Q$ and let $A$ be a $G$-module, meaning that $G$ acts on the abelian group $A$ such that the identity of $G$ acts trivially on every $a \in A$, the action of $\sigma , \tau \in G$ satisfy  $\sigma (a_1 + a_2) = \sigma (a_1) + \sigma (a_2)$ and $\tau ( \sigma (a)) = (\tau \sigma )(a)$. In addition to the usual definition of a homomorphism $\phi$ between two abelian groups $A$ and $B$, a homomorphism $\phi : A \longrightarrow B$ of $G$-modules $A$ and $B$ must satisfy $\phi (\sigma (a)) = \sigma (\phi (a))$. Recall the definitions
 \begin{eqnarray*}
H^0(G, A) & = & \{ a \in A \mid \text{ for all } \sigma \in G, \sigma (a) = a \} , \\
B^1 (G, A) & = & \{ f: G \rightarrow A \mid \text{ there exists } a \in A \text{ such that for all } \sigma \in G, f (\sigma ) = \sigma (a) - a  \} , \\
Z^1 (G, A) & = & \{ f: G \rightarrow A \mid f(\tau \sigma ) = \tau (f(\sigma ))+f(\tau ) \} , \\ 
H^1(G, A) & = & \frac{Z^1(G,A)}{B^1(G,A)} . 
\end{eqnarray*}
We use the notations $\tau (a)$ and $a^{\tau }$ interchangeably to mean the action of $\tau \in G$ on $a \in A$.

The action of $G$ on $\mathcal{F}^2$ is trivial: take any algebraic integer point $q = (t, u)$ of the conic $A^2 t^2 + B t u + C u^2 = 1$, by applying the action $\tau \in G$ to $q$ we obtain another point of this conic, the binary quadratic form $(A^2, B, C)$ is unchanged. The action of $G$ on $\Pt{P} = (x, y) \in \mathcal{P}(\overline{\Q})$ is defined as $\Pt{P}^{\tau} = (x^{\tau }, y^{\tau })$ for $\tau \in G$. 

There is an exact sequence of $G$-modules
\begin{equation}\label{theGs} 
\begin{CD} 
1 @>>> \mathcal{P}(\overline{\Z}) @>>> \mathcal{P}(\overline{\Q})_{\text{prim}} @>{\theta }>> \mathcal{F}^2 @>>> 1 .
\end{CD}
\end{equation}

\begin{proof}
We have shown that the map $\theta : \mathcal{P}(\overline{\Q})_{\text{prim}} \rightarrow \mathcal{F}^2$ of Theorem \ref{homneeded} is a homomorphism of abelian groups. To show that $\theta$ is a homomorphism of $G$-modules, we must check that $Q_{\Pt{P}}$ and $Q_{\Pt{P}^{\tau }}$ belong to the same class of $\mathcal{F}^2$. If $\Pt{P} = \Pt{O}$ then this is trivial. Assume that $\Pt{P} \not= \Pt{O}$. By arguing along similar lines to the proof of Lemma \ref{equaldens} observing that $y$ and $y^{\tau }$ share a minimum polynomial, we have $\text{den}(y) = \text{den}(y^{\tau })$ and $N_{\overline{\Q}/\Q}(y) = N_{\overline{\Q}/\Q}(y^{\tau })$ so that $\Pt{P}$ and $\Pt{P}^{\tau }$ have the same denominator $A$.

Let $\beta$ be the ratio of $\Pt{P} \in \mathcal{P}(\overline{\Q})_{\text{prim}}$. Then there exist algebraic integers $t, u$ satisfying $\Pt{P} = \bigl( \frac{A^2 t + \beta u}{A}, \frac{u}{A} \bigr)$ and $Q_{\Pt{P}}(t, u) = 1$. It is clear that $Q_{\Pt{P}}(t^{\tau }, u^{\tau }) = 1$ and $\Pt{P}^{\tau } = \bigl( \frac{A^2 t^{\tau } + \beta u^{\tau }}{A}, \frac{u^{\tau }}{A} \bigr)$ maps to $Q_{\Pt{P}}$ under $\theta $. This shows that $Q_{\Pt{P}}$ and $Q_{\Pt{P}^{\tau }}$ belong to the same class of $\mathcal{F}^2$.
\end{proof}
 
\begin{theorem}\label{Hzerof} 
Let $G = \text{Gal}(\overline{\Q}/\Q)$. Then $H^0(G, \mathcal{P}(\overline{\Q})_{\text{prim}}) = \mathcal{P}(\Q)$ and $H^0(G, \mathcal{F}^2) = \mathcal{F}^2$. 
\end{theorem}

\begin{proof} 
Let $(x,y) \in \mathcal{P}(\overline{\Q})_{\text{prim}}$. If for all $\tau \in G$, $(x^{\tau },y^{\tau }) = (x,y)$ then clearly $x,y \in \Q$ so $H^0(G, \mathcal{P}(\overline{\Q})_{\text{prim.}}) \leq  \mathcal{P}(\Q)$. If $(x,y) \in \mathcal{P}(\Z)$ then $(x,y) \in \mathcal{P}(\Q)$. For any other $(x,y) \in \mathcal{P}(\Q)$ we may put $x = \frac{B}{A}$, $y = \frac{C}{A}$ where $A$ and $C$ are relatively prime integers and thus $B \equiv (B C^{-1})C \pmod{A^2}$ so that $\beta \equiv B C^{-1} \pmod{A^2}$ and $(x,y) \in \mathcal{P}(\overline{\Q})_{\text{prim}}$ so $\mathcal{P}(\Q) \leq H^0(G, \mathcal{P}(\overline{\Q})_{\text{prim}})$. The action of $G$ on $\mathcal{F}^2$ is trivial. Therefore $\mathcal{F}^2$ is fixed by $G$. 
\end{proof}

\begin{theorem} 
There is an exact sequence 
\begin{equation} \label{com}
\begin{CD} 
1 @>>> \mathcal{P}(\Z ) @>>> \mathcal{P}(\Q ) @>{\Red{\theta }}>> \mathcal{F}^2 @>{\pi \phi }>> \text{Cl}^+(K)^2 @>>> 1 
\end{CD}
\end{equation} 
\end{theorem}

\begin{proof}
By Theorem \ref{homneeded}, $\pi \phi$ is a surjective homomorphism. The kernel of $\pi \phi$ is the set of all $[Q]_{\mathcal{F}}$ such that there exist rational integers $t,u$ satisfying $Q(t, u) = 1$. Let $\Red{\theta }$ be the restriction of $\theta$ to $\mathcal{P}(\Q)$. Then $\Red{\theta } : \Pt{P} \mapsto [Q_{\Pt{P}}]_{\mathcal{F}}$. Clearly $\Red{\theta } (\mathcal{P}(\Q)) \subseteq \text{ker}(\pi \phi )$. Conversely take any integers $t, u$ satisfying $Q(t, u) = 1$ where $Q \in \mathcal{F}^2$ and apply $\varphi^{-1}(t, u) \in \mathcal{P}(\Q)$. This shows that $\text{ker}(\pi \phi ) \subseteq \Red{\theta } (\mathcal{P}(\Q))$. Now we must show that $\text{ker}(\Red{\theta }) = \mathcal{P}(\Z)$. If $(x,y) \in \mathcal{P}(\Z)$ then $\text{den}(y) = 1$, $\beta = 0$, so $\Red{\theta } : (x,y) \mapsto [Q_0]_{\mathcal{F}}$. Conversely if $(x, y) \in \text{ker}(\Red{\theta })$ then $\Red{\theta } : (x,y) \mapsto [Q_0]_{\mathcal{F}}$ so that $\text{den}(y) = 1$ and $\beta = 0$ so $(x,y) \in \mathcal{P}(\Z)$.
\end{proof}

\begin{theorem}\label{mainres} 
Let $G = \text{Gal}(\overline{\Q}/\Q)$. There is a group isomorphism 
\begin{equation} 
\text{Cl}^+(K)^2 \simeq \text{ker}\bigl( H^1(G,\mathcal{P}(\overline{\Z})) \rightarrow H^1(G,\mathcal{P}(\overline{\Q})) \bigr) 
\end{equation} 
\end{theorem}

\begin{proof} 
Taking Galois cohomology on the exact sequence Eqn. \eqref{theGs}, we have the exact sequence 
$$ \begin{CD}  
1 @>>> H^0(G, \mathcal{P}(\overline{\Z})) @>>> H^0(G,\mathcal{P}(\overline{\Q})_{\text{prim}}) @>>> H^0(G, \mathcal{F}^2) @>>>   \\     
  @.   @.    @.     @.     \\    
  @. H^1(G, \mathcal{P}(\overline{\Z})) @>>> H^1(G,\mathcal{P}(\overline{\Q})_{\text{prim}}) @>>> H^1(G, \mathcal{F}^2) @>>> \dots \\
\end{CD}$$
By Theorem \ref{Hzerof} we obtain the exact sequence
$$ \begin{CD}  
1 @>>> \mathcal{P}(\Z) @>>> \mathcal{P}(\Q) @>>> \mathcal{F}^2 @>>>   \\     
  @.   @.    @.     @.     \\    
  @. H^1(G, \mathcal{P}(\overline{\Z})) @>>> H^1(G,\mathcal{P}(\overline{\Q})_{\text{prim}}) @>>> H^1(G, \mathcal{F}^2) @>>> \dots \\
\end{CD}$$
 The homomorphism $H^1(G,\mathcal{P}(\overline{\Q})_{\text{prim}}) \rightarrow H^1(G,\mathcal{P}(\overline{\Q}))$ is clearly injective and thus we have the short exact sequence
\begin{equation}\label{fin} 
\begin{CD}  
1 @>>> \frac{\mathcal{P}(\Q )}{\mathcal{P}(\Z )} @>>> \mathcal{F}^2 @>>> \text{ker}\bigl( H^1(G,\mathcal{P}(\overline{\Z})) \rightarrow H^1(G,\mathcal{P}(\overline{\Q})) \bigr) @>>> 1 .     
\end{CD}
\end{equation} 
The result follows by comparing the exact sequences \eqref{com} and \eqref{fin}. 
\end{proof}

This completes one of our goals:

\begin{definition}
The Tate-Shafarevich group for the Pell conic $\mathcal{P}$ over integers may be defined as 
\begin{equation}\label{cohomdefSha} 
\Sha (\mathcal{P} / \Z ) = \text{ker}\bigl( H^1(G,\mathcal{P}(\overline{\Z})) \rightarrow H^1(G,\mathcal{P}(\overline{\Q})) \bigr) 
\end{equation}
\end{definition}

\section{Another proof using principal homogeneous spaces}

The group $H^1(G, \mathcal{P}(\overline{\Z}))$ may be identified with principal homogeneous spaces for $\mathcal{P}$ over the integers, and $H^1(G, \mathcal{P}(\overline{\Q}))$ may be identified with principal homogeneous spaces for $\mathcal{P}$ over the rational numbers. 
We discuss the addition map $\mu$ defined by automorphs of the binary quadratic form $Q$. The observation that the pair $( Q(t, u) = 1, \mu )$ forms a principal homogeneous space for Pell conics $\mathcal{P}$ over integers and rational numbers is due to Lemmermeyer.

\begin{definition} 
A principal homogeneous space for the Pell conic $\mathcal{P} : Q_0(x, y) = 1$ over an abelian group $\mathcal{A}$ is a pair $( \mathcal{T}, \mu )$ where $\mathcal{T}/\mathcal{A}$ is a smooth curve and $$\mu : \mathcal{T}(\mathcal{A}) \times \mathcal{P}(\mathcal{A}) \longrightarrow \mathcal{T}(\mathcal{A})$$ is a morphism defined over $\mathcal{A}$ satisfying 
\begin{enumerate} 
\item $\mu (q,\Pt{O}) = q$. 
\item $\mu (\mu (q, \Pt{P}_1),\Pt{P}_2 ) = \mu ( q, \Pt{P}_1 \Padd \Pt{P}_2 )$. 
\item For all $q_1, q_2 \in \mathcal{T}(M)$, there is a unique $\Pt{P} \in \mathcal{P}(\mathcal{A})$ such that $\mu ( q_1, \Pt{P} ) = q_2$.
\end{enumerate} 
\end{definition}

\begin{definition} 
Fix a binary quadratic form $Q = (A^2, B , C ) \in \mathfrak{F}^2$ and let $\mathcal{T}(\overline{\Z}) = \{ t,u \in \overline{\Z} \mid Q(t, u) = 1 \}$. Define the maps 
\begin{align}
\label{add}   \mu & : \mathcal{T}(\overline{\Z}) \times \mathcal{P}(\overline{\Z}) \rightarrow \mathcal{T}(\overline{\Z}), & \mu ( q, \Pt{P} ) = M_{\Pt{P}} \utilde{q} , \\ 
\label{subtr} \nu & : \mathcal{T}(\overline{\Z}) \times \mathcal{T}(\overline{\Z}) \rightarrow \mathcal{P}(\overline{\Z}), & \nu ( q_2, q_1 ) = L_{Q, q_1} \utilde{q_2} .
\end{align}
\end{definition}

\begin{lemma}
The pair $(\mathcal{T}: Q(t, u) = 1, \mu )$ is a principal homogeneous space for the Pell conic $\mathcal{P}$ over both integers and over rational numbers, where $Q = (A^2, B , C )\in \mathfrak{F}^2$ and $\mu $ is as in Eqn. \eqref{add}.
\end{lemma}
We find an analogous result to the following lemma in \cite{Silverman}.
\begin{lemma}\label{silv} 
Let $\Pt{P}, \Pt{P}_1, \Pt{P}_2 \in \mathcal{P}(\overline{\Q})_{\text{prim}}$ where the points $\Pt{P}_1 = (x_1, y_1)$ and $\Pt{P}_2= (x_2, y_2)$ have the same denominator and ratio. Let $o = \bigl( \frac{1}{A}, 0 \bigr)$ satisfy $Q_{\Pt{P}}\bigl( \frac{1}{A}, 0 \bigr) = 1$ , and let $q_1 = (t_1, u_1)$ and $q_2 = (t_2, u_2)$ satisfy $Q(t_1, u_1) = 1 = Q(t_2, u_2)$ for some $Q$ in a class of $\mathcal{F}^2$. Then 
\begin{enumerate} 
\item $\mu (o, \Pt{P}) = \varphi(\Pt{P})$, 
\item $\nu (\varphi (\Pt{P}_2), \varphi (\Pt{P}_1)) = \Pt{P}_2 \Psub \Pt{P}_1$, 
\item $\nu (q_2, q_1) = \varphi^{-1}(q_2) \Psub \varphi^{-1}(q_1)$. 
\end{enumerate} 
\end{lemma}

\begin{proof}
It is a simple matter to evaluate Eqn.s \eqref{add} and \eqref{subtr}. 
\small
\begin{equation*} 
\mu \bigl( o , \Pt{P} \bigr) = M_{\Pt{P}} \utilde{o} =  \MS_Q \utilde{\Pt{P}} = \varphi (\Pt{P}) .
\end{equation*} 
\begin{equation*} 
\nu (\varphi (\Pt{P}_2), \varphi (\Pt{P}_1))  =  \nu ( \MS_Q \utilde{\Pt{P}_2} , \MS_Q \utilde{\Pt{P}_1} ) =  L_{Q, \MS_Q \utilde{\Pt{P}_1}} \MS_Q \utilde{\Pt{P}_2} =  \MN_{\Pt{P}_1} \utilde{\Pt{P}_2} =  \Pt{P}_2 \Psub \Pt{P}_1 .
\end{equation*} 
\begin{equation*}
\nu ( q_2, q_1 )  = L_{q_1} \utilde{q_2} = \MS_Q^{-1} \utilde{q_2} \Psub \MS_Q^{-1} \utilde{q_1} =  \varphi^{-1}(q_2) \Psub \varphi^{-1}(q_1) .
\end{equation*}
\normalsize
\end{proof}

\begin{theorem}\label{altlemma} 
Define\footnote{Lemmermeyer \cite{LemNewBook} does something similar in his book starting from a different set, from which the author got the idea.} a map $\xi : \mathcal{P}(\overline{\Q})_{\text{prim.}} \rightarrow  H^1(G, \mathcal{P}(\overline{\Z}))$ by $\xi : \Pt{P} \mapsto \{ f : \tau \mapsto \nu(q^{\tau},q) \}$, where $q = (t, u)$ is any algebraic integer point satisfying $Q_{\Pt{P}}(t,u) = 1$. There is an exact sequence 
$$\begin{CD} 
1 @>>> \mathcal{P}(\Q ) \oplus \mathcal{P}(\overline{\Z}) @>>> \mathcal{P}(\overline{\Q})_{\text{prim}} @>{\xi}>> H^1(G,\mathcal{P}(\overline{\Z})) @>{\psi}>> H^1(G,\mathcal{P}(\overline{\Q})_{\text{prim}}) . 
\end{CD}$$ 
\end{theorem} 

\begin{proof} 
The proof involves expanding $\xi (\Pt{P}_1 \Padd \Pt{P}_2)$ and $\xi (\Pt{P}_1) \cdot \xi (\Pt{P}_2)$ and comparing their cohomology classes. Let $\Pt{P}_1, \Pt{P}_2 \in \mathcal{P}(\overline{\Q})_{\text{prim}}$, $\Pt{P}_3 = \Pt{P}_1 \Padd \Pt{P}_2$, and $q_3 = (t_3, u_3) = (t_1, u_1) \circ (t_2, u_2)$ satisfy $Q_{\Pt{P}_3}(t_3, u_3) = 1$ over $\overline{\Z}$. Recall that 
\begin{eqnarray*}
t_3 & = & e^2 t_1 t_2 + \frac{e^2}{A_2^2} \beta_{23}^- t_1 u_2 + \frac{e^2}{A_1^2} \beta_{13}^- t_2 u_1 + \frac{e^2}{A_1^2 A_2^2} (\beta^{\times} - \beta_3 \beta^+ ) u_1 u_2 , \\
u_3 & = & \frac{A_1^2}{e^2} t_1 u_2 + \frac{A_2^2}{e^2} t_2 u_1 + \frac{\beta^+}{e^2} u_1 u_2 .
\end{eqnarray*}

\begin{eqnarray*} 
\xi ( \Pt{P}_1 \Padd \Pt{P}_2 ) & = & \bigl\{ f_3 : \tau \mapsto \nu ( q_3^{\tau}, q_3 ) \bigr\} , \\ 
                            & = & \bigl\{ f_3 : \tau \mapsto L_{q_3} \utilde{q_3^{\tau }} \bigr\} . \\ 
\xi \bigl( \Pt{P}_1 \bigr) \cdot \xi \bigl( \Pt{P}_2 \bigr) & = & \{ f_1 : \tau \mapsto \nu ( q_1^{\tau}, q_1 ) \} \cdot \{ f_2 : \tau \mapsto \nu ( q_1^{\tau}, q_1) \}  \\ 
                                                            & = & \{ f_1 \cdot f_2 : \tau \mapsto \nu ( q_1^{\tau}, q_1 ) \Padd \nu ( q_2^{\tau}, q_2 ) \} , \\
                                                            & = & \Bigl\{ f_1 \cdot f_2 : \tau \mapsto L_{Q_1, q_1} \utilde{q_1^{\tau }} \Padd L_{Q_2, q_2} \utilde{q_2^{\tau }}  \Bigr\} .
\end{eqnarray*}
We must compare the points $f_1\cdot f_2(\tau )$ and $f_3(\tau )$. Let
\begin{eqnarray*}
\Lambda & = & A_1^2 A_2^2 t_1 t_2 + A_1^2 \beta_2 t_1 u_2 + A_2^2 \beta_1 t_2 u_1 + \beta^{\times} u_1 u_2 , \\
    \Xi & = & A_1^2 t_1 u_2 + A_2^2 t_2 u_1 + \beta^+ u_1 u_2 .
\end{eqnarray*}

Using the identity $\Pt{P}_1 \Padd \Pt{P}_2 = \MN_{\Pt{P}_1} \utilde{\Pt{P}_2}$, we write $f_1\cdot f_2(\tau )$ 
\begin{eqnarray*} 
f_1\cdot f_2(\tau ) & = & \smatr{(A_1^2 t_1 + \beta_1 u_1 + \Si u_1) t_1^{\tau } + (\beta_1 t_1 + \gamma_1 u_1) u_1^{\tau }}{\Em (t_1 u_1^{\tau } - t_1^{\tau} u_1)}{(t_1 u_1^{\tau } - t_1^{\tau } u_1)}{(A_1^2 t_1 + \beta_1 u_1) t_1^{\tau } + (\beta_1 t_1 + \Si t_1 +\gamma_1 u_1 ) u_1^{\tau }} L_{Q_2, q_2} \utilde{q_2^{\tau }} . \\
f_3(\tau ) & = & \smatr{A_3^2 t_3 + (\beta_3 + \Si ) u_3}{\beta_3 t_3 + \gamma_3 u_3}{-u_3}{t_3} \utilde{q_3^{\tau}} , \\
           & = & \smatr{A_3 t_3 + (\beta_3 + \Si ) u_3}{\beta_3 t_3 + \gamma_3 u_3}{-u_3}{t_3} \smatr{e^2 t_1^{\tau } + \frac{e^2}{A_1^2}\beta_{13}^- u_1^{\tau }}{\frac{e^2}{A_2^2} \beta_{23}^- t_1^{\tau } + \frac{e^2}{A_1^2 A_2^2}(\beta^{\times }-\beta_3 \beta^+ ) u_1^{\tau }}{\frac{A_1^2}{e^2} u_1^{\tau }}{\frac{A_1^2}{e^2} t_1^{\tau } + \frac{\beta^+}{e^2} u_1^{\tau }} \utilde{q_2^{\tau}} , \\
           & = & \Bigl( \smatr{\frac{1}{e^2}(\Lambda + \Si \Xi )}{\frac{-e^2 \Em }{A_1^2 A_2^2} \Xi }{\frac{-1}{e^2} \Xi }{\frac{e^2}{A_1^2 A_2^2} \Lambda } + \frac{e^2 \beta_3}{A_1^2 A_2^2} \smatr{0}{\Lambda + \Si \Xi }{0}{-\Xi } \Bigr) \\
          &   & \times \Bigl( \smatr{\frac{e^2}{A_1^2}(A_1^2 t_1^{\tau } + \beta_1 u_1^{\tau })}{\frac{e^2}{A_1^2 A_2^2}(A_1^2 \beta_2 t_1^{\tau } + \beta^{\times} u_1^{\tau })}{\frac{A_2^2}{e^2} u_1^{\tau }}{\frac{1}{e^2}(A_1^2 t_1^{\tau } + \beta^+ u_1^{\tau })} -\frac{e^2 \beta_3}{A_1^2 A_2^2} \smatr{A_2^2 u_1^{\tau }}{A_1^2 t_1^{\tau } + \beta^+ u_1^{\tau }}{0}{0} \Bigr) \utilde{q_2^{\tau}} , \\
          & = & \smatr{(\Lambda + \Si \Xi )}{\frac{- \Em }{A_1^2 A_2^2} \Xi }{- \Xi }{\frac{1}{A_1^2 A_2^2} \Lambda } \smatr{\frac{1}{A_1^2}(A_1^2 t_1^{\tau } + \beta_1 u_1^{\tau })}{\frac{1}{A_1^2 A_2^2}(A_1^2 \beta_2 t_1^{\tau } + \beta^{\times} u_1^{\tau })}{A_2^2 u_1^{\tau }}{(A_1^2 t_1^{\tau } + \beta^+ u_1^{\tau })} \utilde{q_2^{\tau}} , \\
          & = & f_1 \cdot f_2(\tau ) .
\end{eqnarray*}
Therefore $f_1 \cdot f_2$ and $f_3$ belong to the same cohomology class of $H^1(G,\mathcal{P}(\overline{\Z}))$ so that $\xi (\Pt{P}_1 \Padd \Pt{P}_2) = \xi (\Pt{P}_1)\cdot \xi (\Pt{P}_2)$ and $\xi$ is a homomorphism. Now $\Pt{P} \in \text{ker }\xi $ if and only if $Q_{\Pt{P}}(t, u) = 1$ has an integer solution, if and only if $\Pt{P} \in \text{ker}(\pi \lambda_K ) = \mathcal{P}(\Q) \oplus \mathcal{P}(\overline{\Z})$. 

The homomorphism $\psi$ is given by $\psi : f\cdot B^1(G,\mathcal{P}(\overline{\Z})) \mapsto f\cdot B^1(G,\mathcal{P}(\overline{\Q})_{\text{prim}})$. It remains to show that the image of $\xi$ is equal to the kernel of $\psi$. Now $\xi (\Pt{P}) = \{ f: \tau \mapsto \nu(q^{\tau},q) \}$ where $q$ is any $q =(t, u)$ satisfying $Q_{\Pt{P}}(t,u) = 1$ over $\overline{\Z}$. There exists a $\Pt{P} \in \mathcal{P}(\overline{\Q})_{\text{prim}}$, namely $\Pt{P} = \varphi^{-1} (t,u) $, such that $q = \varphi (\Pt{P})$. So $f(\tau) = \nu (\varphi (\Pt{P})^{\tau},\varphi (\Pt{P}))$. Since $\Pt{P} \in \mathcal{P}(\overline{\Q})_{\text{prim}}$, $$f(\tau) = \nu (\varphi (\Pt{P}^{\tau}),\varphi (\Pt{P})) = \Pt{P}^{\tau} \Psub \Pt{P}$$ by Lemma \ref{silv}. So $f \in B^1(G,\mathcal{P}(\overline{\Q})_{\text{prim}})$. This shows that $\text{Im }(\xi) \leq \text{ker }\psi$. Conversely if $\Pt{P} \in \mathcal{P}(\overline{\Q})_{\text{prim}}$, $g:  \tau \mapsto \Pt{P}^{\tau} \Psub \Pt{P} \in B^1(G,\mathcal{P}(\overline{\Q})_{\text{prim}})$. There exist $\Pt{P} \in \mathcal{P}(\overline{\Q})_{\text{prim}}$ and $q = (t, u)$ satisfying $Q_{\Pt{P}}(t, u) = 1$ over $\overline{\Z}$ such that $\Pt{P} = \varphi^{-1}(q)$, these $\Pt{P}$ and $q$ exist because $\Pt{P}$ is an element of the image of $\varphi^{-1}$. Therefore 
\begin{equation*} 
g(\tau) = \varphi^{-1} (q)^{\tau} \Psub \varphi^{-1}(q) = \varphi^{-1} (q^{\tau}) \Psub \varphi^{-1}(q) = \nu( q^{\tau}, q) , 
\end{equation*} 
so that $g(\tau) \in \text{Im }(\xi)$. This completes the proof that $\text{Im }(\xi) = \text{ker }\psi$. 
\end{proof}
Again we obtain $\text{Cl}^+(K)^2 \simeq \text{ker}\bigl( H^1(G,\mathcal{P}(\overline{\Z})) \rightarrow H^1(G,\mathcal{P}(\overline{\Q})) \bigr) $ as a consequence of Theorem \ref{altlemma}, as the result is compared with the exact sequence 
$$\begin{CD} 
1 @>>>  \mathcal{P}(\Q ) \oplus \mathcal{P}(\overline{\Z}) @>>> \mathcal{P}(\overline{\Q})_{\text{prim}} @>>> \text{Cl}^+(\Delta )^2 @>>> 1 ,
\end{CD}$$
by Theorem \ref{homneeded}.

We speculate that the map $\zeta $ in the exact sequence 
$$\begin{CD} 
1 @>>> \text{Cl}^+(\Delta )^2 @>{\zeta }>> H^1(G,\mathcal{P}(\overline{\Z})) @>>> H^1(G,\mathcal{P}(\overline{\Q}))
\end{CD}$$
is given by $\zeta : [Q]_{\Delta } \mapsto \{ f : \tau \mapsto \nu(q^{\tau},q) \} $, where $q = (t, u)$ is any algebraic integer point satisfying $Q(t, u) = 1$ and $\nu$ is as in Eqn. \eqref{subtr}.

\section*{Acknowledgments}

The author thanks Victor Scharaschkin for doctoral supervision of which this project has been a part of, and supported by the University of Queensland. The author sincerely appreciates several suggestions from Franz Lemmermeyer on presenting some of the topics covered in this article, including the role of automorphs in the addition map for torsors of Pell conics. The author is grateful for having had the opportunity to read {\em Binary quadratic forms} \cite{LemNewBook}. This book has been very useful and the reader will find some of our notation to be similar.  

\bibliographystyle{amsplain}

\end{document}